\documentclass{amsart}
\numberwithin{equation}{section}
\usepackage{amssymb}
\usepackage[all]{xy}
\let\cal\mathcal

\def\Cscr{{\cal C}}
\def\Dscr{{\cal D}}
\def\Escr{{\cal E}}
\def\Fscr{{\cal F}}

\def\Oscr{{\cal O}}

\let\blb\mathbb

\def \PP{{\blb P}}

\def \ZZ{{\blb Z}}

\def\Mod{\operatorname{Mod}}
\def\mod{\operatorname{mod}}

\def\qgr{\operatorname{qgr}}

\def\gr{\operatorname{gr}}

\def\gr{\operatorname {gr}}

\def\Ext{\operatorname {Ext}}
\def\Hom{\operatorname {Hom}}

\def\im{\operatorname {im}}

\def\coker{\operatorname {coker}}
\def\ker{\operatorname {ker}}

\def\Tor{\operatorname {Tor}}

\def\r{\rightarrow}

\DeclareMathOperator{\Pro}{Pro}
\DeclareMathOperator{\Ind}{Ind}

\let\dirlim\injlim
\let\invlim\projlim

\newtheorem{lemma}{Lemma}[section]
\newtheorem{proposition}[lemma]{Proposition}
\newtheorem{theorem}[lemma]{Theorem}

\newtheorem{lemmas}{Lemma}[subsection]
\newtheorem{propositions}[lemmas]{Proposition}
\newtheorem{theorems}[lemmas]{Theorem}

\theoremstyle{definition}

\newtheorem{definition}[lemma]{Definition}

\newtheorem{definitions}[lemmas]{Definition}

{

\newtheorem{step}{Step}

}

\theoremstyle{remark}

\newdimen\uboxsep \uboxsep=1ex
\def\uboxn#1{\vtop to 0pt{\hrule height 0pt depth 0pt\vskip\uboxsep
\hbox to 0pt{\hss #1\hss}\vss}}

\def\uboxs#1{\vbox to 0pt{\vss\hbox to 0pt{\hss #1\hss}
\vskip\uboxsep\hrule height 0pt depth 0pt}}

\title{Notes on formal deformations of abelian categories}
\author{Michel Van den Bergh}
\address[Michel Van den Bergh]{Departement WNI, Hasselt University, Agoralaan, 3590 Diepenbeek, Belgium}
\email{michel.vandenbergh@uhasselt.be} 

\thanks{The author is a director of research at the FWO}

\def\Pro{\operatorname{Pro}}

\keywords{Deformation theory, abelian categories}
\subjclass{Primary 13D10, 14A22, 18E1} 
\begin{document}
\begin{abstract} In these notes we provide the foundation for the deformation
theoretic parts of arXiv:0807.375  and arXiv:math/0102005. 
\end{abstract}
\maketitle
\section{Introduction}
In these notes we provide the foundation for the deformation theoretic
parts of \cite{VdB28,VdB26}.  In \cite{VdB26} we construct
non-commutative analogues of quadrics and in \cite{VdB28} we define
non-commutative $\PP^1$-bundles over commutative varieties. A notable
special case of the latter are non-commutative analogues of Hirzebruch
surfaces.

Indeed \cite{VdB28} contains a proof that any formal deformation of a Hirzebruch
surface (in a suitable sense) is given by a non-commutative
Hirzebruch surface. Similarly the original (privately circulated) version of
\cite{VdB26} contains a proof that any formal deformation of a quadric is
a non-commutative quadric (see \cite[\S
11.2]{VdBSt} for a sketch). I deleted this proof when I
first put the paper on the arXiv (8 years after it was written) since I was
unhappy with the deformation theoretic setup that was used.

Meanwhile a satisfactory infinitesimal deformation theory for abelian
categories has been developed in \cite{lowenvdb2,lowenvdb1}. 
In
the noetherian setting (which is sufficient for the applications we
have in mind) the passage from the infinitesimal context to the
formal context is an application of Jouanolou's results
in \cite{Joua}. Nonetheless Jounalou's expos\'e is written for a different
purpose so some translation
is necessary. After several (not very satisfactory) attempts to
rewrite the deformation theoretic parts of \cite{VdB28,VdB26} using
Jouanolou's language of ``AR-J-adic systems'' I decided that it was better to
write a self contained paper on formal deformations of noetherian abelian
categories, which resulted
in the current paper.
On the purely mathematical level there is very little originality in
what we will do.  Besides Jouanolou's expos\'e we have also borrowed
from \cite{AZ2} (which basically discusses trivial deformations) and
\cite[\S5]{EGA31} (which discusses formal schemes).  On the expository
level we deviate from the aforementioned references by systematically
using Pro-objects instead of adic objects. Pro-objects form a
co-Grothen\-dieck category so in particular they have very well
behaved inverse limits.

\medskip

We now give a more detailed exposition of our setup. Let $R$ be a
commutative noetherian ring and let $J$ be an ideal in $R$. 
\begin{definition}
  Let $\Cscr$ be a noetherian $R$-linear abelian category.  The \emph{completion}
  $\widehat{\Cscr}$ of $\Cscr$ is the full subcategory of
  $\Pro(\Cscr)$ consisting of the pro-objects $M$ over $\Cscr$ such that $M/MJ^n\in
  \Cscr$ for all $n$ and such that the canonical map $M\r \invlim_n
  M/MJ^n$ is an isomorphism.
\end{definition}
We reproduce Jouanolou's proof
(recast in our language) that $\widehat{\Cscr}$ is a noetherian
abelian category (see Proposition \ref{ref-2.2.4-6} below).

There is an obvious exact functor $\Phi:\Cscr\r\widehat{\Cscr}:M\mapsto \invlim_n M/MJ^n$ and we say that
$\Cscr$ is \emph{complete} if this functor is an equivalence. 
Roughly speaking an \emph{$(R,J)$-deformation} of an $R/J$-linear abelian
category $\Escr$ will be a complete $R$-linear category $\Fscr$
together with an equivalence $\Escr\cong\Fscr_{R/J}$ where
$\Fscr_{R/J}$ is the full subcategory of $\Fscr$ consisting of objects
annihilated by $J$. To make this definition work one has to impose
certain flatness conditions. See \S\ref{ref-2.1-1} and \S\ref{ref-3-35} for
more details.

\medskip

So to understand formal deformations of abelian categories we
have to understand completion. We first note that
completion extends to functors. We show that if $(T^i)_i:\Cscr\r
\Dscr$ is $\partial$-functor between noetherian abelian $R$-linear categories
then this functor extends to a $\partial$-functor
$(\hat{T}^i)_i:\widehat{\Cscr}\r \widehat{\Dscr}$ (see Theorem
\ref{ref-2.3.1-15}). This is a slight improvement over \cite{Joua} as
Jouanolou imposes some extra conditions on $T$ which seem to be
superfluous.

\medskip

We use the good behaviour of $\partial$-functors to study
$\Ext$-groups. Let $\Cscr_t$ be the full subcategory of $\Cscr$
consisting of the objects which are annihilated by some power
of~$J$. Then we define the \emph{completed $\Ext$-groups} between objects
$M,N\in \widehat{\Cscr}$ as follows
\begin{equation}
\label{new-1}
`\Ext^i_{\widehat{\Cscr}}(M,N)=\Ext^i_{\Pro(\Cscr_t)}(M,N)
\end{equation}
An alternative point of view to this definition is that we consider
the full subcategory $D_c(\Cscr)$ of $D(\Pro(\Cscr_t))$ of complexes whose
cohomology lies in $\widehat{\Cscr}$. Thus $D_c(\Cscr)$ has a $t$-structure whose heart is
$\widehat{\Cscr}$. 
Then the completed $\Ext$-groups for $M,N\in \widehat{\Cscr}$ may be reinterpreted as
\[
{}`\Ext_{\widehat{\Cscr}}(M,N)=\Hom_{D_c(\Cscr)}(M,N[n])
\]
In the case that $\Cscr$ is the category of torsion $l$-adic constructible
sheaves it would be interesting to compare this derived category
to the standard derived category of $l$-adic sheaves \cite{Behrend1,BBD,Ek}.

\medskip

Obviously $`\Ext_{\widehat{\Cscr}}(-,-)$ is a $\partial$-functor in both arguments
but
apart from this we don't have anything to say about it. However in the
event that $\Cscr_{R/J}$ (the objects in $\Cscr$ annihilated by $J$)
has finitely generated $\Ext$-groups over $R/J$ and $\Cscr$ is ``formally flat'' (see
\S\ref{ref-2.4-18}) then we have the expected formula
\[
`\Ext^i_{\widehat{\Cscr}}(M,N)=\invlim_k \dirlim_l \Ext_{\Cscr_{R/J^l}}^i(M/MJ^l,N/NJ^k)
\]
An important theorem in algebraic geometry is Grothendieck's existence
theorem~\cite{EGA31}.  This theorem extends to the current setting (see also
\cite{AZ2}). In \S\ref{ref-4-36} we introduce the notion of a \emph{strongly
  ample sequence}.\footnote{We could have used ample
  sequences \cite{Polishchuk1} but to have good behaviour of higher
  $\Ext$-groups it is convenient to use a slightly stronger notion.}
By definition a sequence $(O(n))_{n\in\ZZ}$ of objects in a noetherian abelian
category $\Escr$ is strongly ample if the following conditions hold
\begin{itemize}
\item[(A1)] For all  $M\in \Escr$ and for all $n$
there is an epimorphism $\oplus_{i=1}^t O(-n_i)\r M$ with $n_i\ge n$.
\item[(A2)] For all $M\in \Escr$ and for all $i>0$ one has
  $\Ext^i_\Escr(O(-n),M)=0$ for $n\gg 0$.
\end{itemize}
A strongly ample sequence $(O(n))_{n\in \ZZ}$ in $\Escr$ is ample in
the sense of \cite{Polishchuk1}. Hence using the methods of \cite{AZ}
or \cite{Polishchuk1} one obtains $ \Escr\cong \gr(A)/{\mathrm{f.l.}} $ if $\Escr$ is
$\Hom$-finite, where $A$ is the noetherian $\ZZ$-algebra $\oplus_{ij}
\Hom_\Escr(O(-j),O(-i))$.

The following is our version of Grothendieck's existence theorem.
\begin{proposition} (see Proposition \ref{ref-4.1-37}) Assume that $R$ is
  $J$-adically complete.  Let $\Escr$ be an $\Ext$-finite $R$-linear
  noetherian category with a strongly ample sequence $(O(n))_n$. Then
  $\Escr$ is complete and furthermore if $\Escr$ is flat (see \S\ref{ref-2.1-1}) we have for
  $M,N\in\Escr$:
\begin{equation}
\label{ref-1.1-0}
\Ext^i_\Escr(M,N)=`\Ext^i_\Escr(M,N)
\end{equation}
\end{proposition}
The property for a sequence to be strongly ample lifts well under deformations.
\begin{theorem} (an extract of Theorem \ref{ref-4.2-44}) Let $\Dscr$
  be an $R$-deformation of an $\Ext$-finite flat $R/J$-linear
  noetherian abelian category $\Cscr$ and $(O(n))_n$ be a sequence of
  $R$-flat objects in $\Dscr$. Then $(O(n)/O(n)J)_n$ is strongly ample
  in $\Cscr$ if and only if $(O(n))_n$ is strongly ample in $\Dscr$.
\end{theorem}
Many algebraic varieties (e.g.\ Del Pezzo surfaces) have a strongly
ample sequence consisting of exceptional objects. Such a 
sequence can then be lifted to any deformation (see \S\ref{seclb}). This idea is basically
due to Bondal and Polishchuk and is described explicitly in \cite[\S
11.2]{VdBSt}. It was used to define non-commutative quadrics in
\cite{VdB26} and indirectly in the classification of non-commutative
Hirzebruch surfaces in \cite{VdB28}. See also the recent paper
\cite{LDeD}.

\medskip

Let us also mention that a very complete treatment of deformations of algebraic
varieties as ringed spaces (including their derived categories) over $k[[t]]$ has
been given in \cite{HMS}.

\section{Completion of abelian categories}

\subsection{Base extension}
\label{ref-2.1-1}
We recall briefly some notions from \cite{lowenvdb1}. Throughout $R$ will be a commutative
noetherian ring and $\mod(R)$ is its category of finitely generated modules.

Let $\Cscr$ be an $R$-linear abelian category.  Then we have bifunctors $-\otimes_R-:\Cscr\times \mod(R)\r
\Cscr$, $\Hom_R(-,-):\mod(R)^\circ\times \Cscr\r \Cscr$ defined in the usual
way.  These functors may be derived in their $\mod(R)$-argument to
yield bi-delta-functors $\Tor^R_i(-,-)$, $\Ext_R^i(-,-)$.  An object $M\in \Cscr$ is
\emph{$R$-flat} if $M\otimes_R -$ is an exact functor, or equivalently if $\Tor_i^R(M,-)=0$ for
$i>0$.

By
definition (see \cite[\S3]{lowenvdb1}) $\Cscr$ is $R$-\emph{flat} if
$\Tor^R_i$ or equivalently $\Ext^i_R$ is effaceable in its
$\Cscr$-argument for $i>0$.  This implies that $\Tor^R_i$ and
$\Ext_R^i$ are universal $\partial$-functors in both arguments.

If $f:R\r S$ is a morphism of commutative rings and $\Cscr$ is an
$R$-linear abelian category then $\Cscr_S$ denotes the (abelian)
category of objects in $\Cscr$ equipped with an $S$-action.  We
usually refer to objects in $\Cscr_S$ as $(S,\Cscr)$-objects and if
$S$ is graded then we also talk about graded $(S,\Cscr)$-objects.  If
$f$ is surjective then $\Cscr_S$ identifies with the full subcategory
of $\Cscr$ given by the objects annihilated by $\ker f$.  If $R$ is
noetherian and $S$ is module finite over $R$ then the inclusion
functor $\Cscr_S\r \Cscr$ has right and left adjoints given
respectively by $\Hom_R(S,-)$ and $-\otimes_R S$.
\subsection{Completion of noetherian abelian categories}
Below we refer to a pair $(R, J)$ where 
 $R$ is a commutative noetherian ring and $J\subset R$ is an ideal 
as a $J$-\emph{adic noetherian ring}.
Below $(R,J)$ is a $J$-adic noetherian ring. We put $R_n=R/J^n$ and we denote
the $J$-adic completion of $R$ by $\hat{R}$. This is also a noetherian ring.
Using a slight abuse of notation we denote the extended ideal $J\hat{R}$ by $J$.

Recall that an abelian category $\Cscr$ is said to be \emph{noetherian} if it
is essentially small and all objects are noetherian. Below $\Cscr$ is an $R$-linear
noetherian category.

If $\Dscr$ is an essentially small abelian
  category then the category $\Pro(\Dscr)$ of \emph{pro-objects} over $\Dscr$ is the category
whose objects are filtered inverse systems $(M_\alpha)_\alpha$ and 
whose $\Hom$-sets are given by
\begin{equation}
\label{ref-2.1-2}
\Hom_{\Pro(\Dscr)}((M_\alpha)_\alpha,(N_\beta)_\beta)=
\invlim_\beta\dirlim_\alpha\Hom_\Dscr(M_\alpha,N_\beta)
\end{equation}
In other words if we identify $\Dscr$ with the one-object inverse systems in $\Pro(\Dscr)$ then 
$(M_\alpha)_\alpha=\invlim_\alpha M_\alpha$ in $\Pro(\Dscr)$. 
\begin{lemmas} \cite[\S I.8]{SGA41} Assume that $\Dscr$ is an essentially small abelian
  category. Then $\Pro(\Dscr)^{\text{opp}}$ is a Grothendieck category
  and in particular $\Pro(\Dscr)$ has exact filtered inverse limits
  and enough projectives. The natural functor $\Dscr\r \Pro(\Dscr)$ is
  fully faithful exact and its essential image is closed under
  extensions. If $D\in \Dscr$ then $\Hom_\Dscr(-,D)$ sends inverse limits
to direct limits (in other words $D$ is co-finitely presented). 
\end{lemmas}
\begin{definitions}
The \emph{completion} $\widehat{\Cscr}$  of $\Cscr$ is the full subcategory of $\Pro(\Cscr)$
consisting of the objects $M$ such that $M/MJ^n\in \Cscr$ for all $n$ and
such that  the canonical map $M\r \invlim_n M/MJ^n$ is an
isomorphism. 
\end{definitions}
It is easy to see that  $\widehat{\Cscr}$ is a $\widehat{R}$-linear category.
To study objects in $\widehat{\Cscr}$ we need to consider filtrations.
 By
  definition a filtration on an object $M$ of an $R$-linear category
  $\Dscr$ is a descending chain of subobjects $M=\cdots \supset F_0M\supset
  F_1M\supset\cdots$.  The associated graded objects
  $\gr_F M$ is the $\ZZ$-graded object over~$\Dscr$ defined by the
  formal direct sum $\bigoplus_n F_{n}M/F_{n+1}M$.  By $F_J$ we denote the $J$-adic
  filtration. I.e.\ $F_{J,i}M=MJ^i$ for $i\ge 0$ and $F_{J,i}M=M$ for $i\le 0$.

We say that the filtration $F$ is \emph{adapted} to $J$ if
  $(F_iM)J\subset F_{i+1}M$ (see \cite[\S 4.2]{Joua}). In that case $\gr_F M$ is
a graded $(\gr_J R,\Dscr_{R/J})$-object. 

\begin{lemmas} \label{ref-2.2.2-3} If $M$ is a noetherian object in an
  $R$-linear abelian category $\Dscr$ and~$S$ is a positively graded
  noetherian $R$-algebra such that $S_0=R$. Then $M\otimes_R S$ is a noetherian graded
$(S,\Dscr)$-object
\end{lemmas}
\begin{proof} This follows from a variant of  Hilbert's basis theorem. See e.g.\ 
\cite[Thm 5.1.4]{Joua} and \cite[Lemma 4.2.4]{Joua}.
\end{proof}

\begin{lemmas} \label{ref-2.2.3-4} (compare with \cite[Thm 4.2.6]{Joua}) Assume that $K\in \Pro(\Cscr)$ is
  equipped with a $J$-adapted filtration $F$ such that
\begin{enumerate}
\item $\invlim_n K/F_nK=K$.
\item $\gr_F K$ is a noetherian graded $(\gr_J R,\Cscr_{R/J})$-object.
\end{enumerate}
Then $K\in \widehat{\Cscr}$. 
\end{lemmas}
\begin{proof} 
 We follow somewhat the idea of \cite[Lemma
4.2.7]{Joua}.  For any $r\ge 0$ define $F^{(r)}_i K=KJ^{i-r}\cap
F_{i-1} K+F_i K$ (with $J^{i-r}=R$ for $r\ge i$).
Then we have $F_i K\subset F^{(r)}_i K\subset F_{i-1}K$ and $(F^{(r)}_i K)J\subset
F^{(r)}_{i+1} K\subset F_i K$. In other words $\gr^{(r)}_FK\overset{\text{def}}{=}\bigoplus
F^{(r)}_{i+1} K/F_{i+1} K$ is an ascending chain of graded $(\gr_J R,\Cscr_{R/J})$-subobjects of
$\gr_F K$ which must be stationary. Thus there is an $r$ such that for all $i$
\[
KJ^{i-r}\cap
F_{i-1} K+F_i K=KJ^{i-r-1}\cap
F_{i-1} K+F_i K=\cdots=F_{i-1}K
\]
and in particular $F_{i-1}K\subset KJ^{i-r}+F_i K$. Iterating this inclusion and renumbering
we get that there exists an $r$ such that 
\[
F_iK\subset KJ^{i-r}+F_jK
\]
for all $j\ge i$.  Fix $i$ and choose generators $f_1,\ldots,f_p$ for $J^{i-r}$. Then
we get diagrams for $j\ge i$
\begin{equation}
\label{ref-2.2-5}
\xymatrix{
&&0\\
  (K/F_jK)^p\ar^{(f_i)_i}[r] &K/F_j K\ar[r] &K/(F_j K+J^{i-r} K)\ar[r]\ar[u]&0\\
&K\ar[r]\ar[u]&K/F_i K\ar[u]
}
\end{equation}
Using exactness
of filtered inverse limits we get from \eqref{ref-2.2-5} 
\[
\xymatrix{
&&0\\
  K^p\ar^{(f_i)_i}[r] &K\ar[r]&\invlim_j K/(F_j K+J^{i-r} K)\ar[r]\ar[u]&0\\
&K\ar[r]\ar@{=}[u]&K/F_i K\ar[u]
}
\]

and hence from the upper exact sequence we obtain
\[
K/KJ^{i-r}=\invlim_j K/(F_j K+J^{i-r} K)
\]
In other words the identity map $K\r K$ induces a map $K/F_iK\r
K/KJ^{i-r}$ which yields $F_i K\subset KJ^{i-r}$.  

\medskip

The fact that $\gr_F K$ is noetherian implies easily that it has left bounded grading. 
Since $\Cscr$ is closed under extensions inside $\Pro(\Cscr)$ it follows that $K/F_iK\in\Cscr$
for all~$i$.  Furthermore since $\Cscr$ is an abelian subcategory of $\Pro(\Cscr)$ it
is also closed under $-\otimes_RM$ for $M\in \mod(R)$. 

Hence $K/KJ^i=(K/F_{i+r}K)
\otimes_R R/J^i\in \Cscr$.  Furthermore since the $J$-adic filtration and the
$F$-filtration are cofinal we also get $\invlim_n K/KJ^n=\invlim_n K/F_nK=K$.
This shows that indeed $K\in \widehat{\Cscr}$. 
\end{proof}

\begin{propositions} \label{ref-2.2.4-6} (compare with \cite[Thm 5.2.3]{Joua}) $\widehat{\Cscr}$ is a noetherian abelian subcategory of $\Pro(\Cscr)$. 
\end{propositions}
\begin{proof} 
We first prove that $\widehat{\Cscr}$ is an abelian
  subcategory of $\Pro(\Cscr)$. It is obviously closed under cokernels
  (using the exactness of $\invlim$ and right exactness of $-\otimes_R R/J^n$) so we must prove it is closed
  under kernels.

Let
\[
0\r K\r M\r N
\]
be an exact sequence in $\Pro(\Cscr)$ with $M$, $N\in
\widehat{\Cscr}$. We must prove $K\in \widehat{\Cscr}$.  Put
$F_iK\overset{\text{def}}{=}MJ^{i}\cap K\supset KJ^i$.  This is a filtration on $K$
which is  adapted to $J$. Furthermore we have exact sequences
\begin{equation}
\label{ref-2.3-7}
0\r K/F_nK\r M/MJ^n\r N/NJ^n
\end{equation}
By exactness of filtered inverse limits we deduce $K=\invlim_n K/F_nK$. 
Furthermore we obtain exact sequences
\[
0\r \gr_F K\r \gr_{F_J} M\r \gr_{F_J} N
\]
Since $M/MJ\in \Cscr$ it follows from Lemma \ref{ref-2.2.2-3} that $\gr_{F_J}M$ is 
a noetherian graded $(\gr_J R,\Cscr_{R/J})$ object. 
Hence $\gr_F K$ is also a noetherian graded $(\gr_J
R,\Cscr_{R/J})$-object.  By Lemma \ref{ref-2.2.3-4} we conclude $K\in\widehat{\Cscr}$.

It remains to show that $\widehat{\Cscr}$ is noetherian.  Since any
object $M$ in $\widehat{\Cscr}$ satisfies $M=\invlim_n M/J^nM$ and the
category of $\ZZ$-indexed inverse systems over $\Cscr$ is essentially small it
follows that $\widehat{\Cscr}$ is essentially small as well. Thus it remains to
show that any $M\in \widehat{C}$ is noetherian. 

Let $N\hookrightarrow M$ be a subobject of $M$ in $\widehat{\Cscr}$. Put
$F_n N=N\cap MJ^n$. Then $N/F_n N$ is the image of $N/NJ^n\r M/MJ^n$ and so
it lies in $\Cscr$. Furthermore taking the inverse limits of the maps
\[
\xymatrix{
N/NJ^n\ar@{->>}[r]&N/F_n N\ar@{^(->}[r] & M/MJ^n
}
\]
and using exactness of filtered inverse limits we get $N=\invlim_n N/F_n N$.  Now
assume that we have inclusions $N_1\subset N_2\subset M$ in $\widehat{\Cscr}$ such that if
we equip $N_1$, $N_2$ with the filtrations induced from the $J$-adic
filtration on $M$ then the map $\gr_F N_1\r \gr_F N_2$ is an
isomorphism. We claim that then necessarily $N_1=N_2$. Indeed from the
five lemma we obtain $N_1/F_i N_1=N_2/F_i N_2$.  It then suffices to
take inverse limits.

Now let $M^{(r)}\subset M$ be an ascending chain of subobjects and equip them
with the filtrations induced from the $J$-adic filtration on $M$. As indicated above
$\gr_{F_J} M$ is a noetherian graded object over $(\gr_J R,\Cscr_{R/J})$ and hence
the chain $(\gr_F M^{(r)})_r$ is stationary.  By the discussion in the previous paragraph
the chain $(M^{(r)})_r$ is stationary as well. 
\end{proof}
We may compare our definition of $\widehat{\Cscr}$ with the notion of \emph{$J$-adic inverse systems.}
\begin{definitions} (see \cite[\S3.1]{Joua}) Let $\Cscr$ be an $R$-linear noetherian abelian
  category. The category of \emph{$J$-adic inverse systems} $\check{\Cscr}$ over $\Cscr$ is
  defined as the full subcategory of inverse systems $(M_n,\phi_n)$
  over $\Cscr$ such that $M_nJ^n=0$ and such that the transition maps
  $\phi_n:M_n\r M_{n-1}$ induce isomorphisms $M_{n}/M_{n}J^{n-1}\r M_{n-1}$.
\end{definitions}
\begin{propositions} \label{ref-2.2.6-8} The functor
\[
\Sigma:\widehat{\Cscr}\r \check{\Cscr}:M\mapsto (M/MJ^n)_n
\]
is an equivalence of categories. Its inverse is given by
\[
\Psi:\check{\Cscr}\r \widehat{\Cscr}: (N_n)_n\mapsto \invlim_n N_n
\]
\end{propositions}
\begin{proof} We first show that $\Psi$ is well defined.  Let
  $(N_n)_n\in \check{\Cscr}$ and let $N$ be its inverse limit in $\Pro(\Cscr)$. Using
  exactness of filtered inverse limits in $\Pro(\Cscr)$ we get
\[
N/NJ^i=(\invlim_n N_n)\otimes_R R/J^i=\invlim_n (N_n/N_n J^i)=N_i
\]
Thus we have indeed $N=\invlim_i N_i=\invlim_i N/NJ^i$.  From this reasoning we
also get $\Sigma\Psi(N_i)_i=(N_i)_i$. 

The fact that $\Psi\Sigma$ is the identity is by definition. 
\end{proof}

The following easy result motivates the definition of $\widehat{\Cscr}$. 
\begin{propositions} \label{ref-2.2.7-9} One has $\mod(R)\,\hat{}=\mod(\hat{R})$.
\end{propositions}
\begin{proof} In the proof we must distinguish between inverse limits
  in $\Mod(R)$ and $\Pro(\mod(R))$. Therefore we will temporarily
  denote the latter by ${}^p\!\invlim$.

Let ${}^p\!\hat{R}$ be the object of $\mod(R)\,\hat{}$ given by ${}^p\!\invlim_n
R/J^n$. Its endomorphism ring is equal to $\invlim_n R/J^n=\hat{R}$. It
suffices to prove that ${}^p\!\hat{R}$ is a projective generator of
$\mod(R)\,\hat{}$.

We first show that ${}^p\!\hat{R}$ is projective. Let $M\in \mod(R)\,\hat{}$. Then
\begin{align*}
\Hom_{\Pro(\mod(R))}({}^p\!\hat{R},M)&=\invlim_n\Hom_{R/J^n}(R/J^n,M/MJ^n)\\
&=
\invlim_n M/MJ^n
\end{align*}
Hence we must prove that $M\mapsto \invlim_n M/MJ^n$ is exact. Now let 
\[
0\r K\r M\r N\r 0
\]
be an exact sequence in $\mod(R)\,\hat{}$. By lemma \ref{ref-2.2.8-10} below we have that
\[
0\r K/KJ^n \r M/MJ^n\r N/NJ^n\r 0
\]
is exact up to essentially zero systems. From this one easily deduces that
its inverse limit is exact. 

Now we prove that ${}^p\!\hat{R}$ is a generator. Let $M$ be a object of
$\mod(R)\,\hat{}$. Choose $R/J$-generators for $M/JM\in \mod(R/J)$
and lift those to $R/J^n$-generators for $M/MJ^n\in\mod(R/J^n)$. By Nakayama we
get compatible epimorphisms $(R/J^n)^t\r M/MJ^n$ for some fixed $t$.

Taking inverse limits we obtain an epimorphism
${}^p\!\hat{R}^t\r M$
and we are done. 
\end{proof}
The following lemma was used. 
\begin{lemmas} \label{ref-2.2.8-10} Assume that $\Dscr$ is an $R$-linear abelian
category and  $M\subset N$ is an inclusion of noetherian objects in
  $\Dscr$. Then these objects satisfy the Artin-Rees condition in the sense
that there exists an $r$ such that for all $i$ we have
  $NJ^{n+r}\cap M\subset MJ^n$.
\end{lemmas}
\begin{proof}
  This is proved in the standard way. 
Let $\tilde{R}=\bigoplus_{n\ge
    0} J^n$ be the Rees ring of~$R$

The graded ring $\tilde{R}$ is finitely generated over $R$ and it follows from
Lemma \ref{ref-2.2.2-3} that $\bigoplus_i NJ^i$ is a noetherian graded object
  over $(\tilde{R},\Dscr)$.
  Hence so is $C=\bigoplus_i NJ^i\cap M$. 

  Inside $C$ we have an ascending chain of subobjects
  $C^{(r)}=\bigoplus_i NJ^i \cap MJ^{i-r}$ (with $J^p=R$ for $p\le
  0$) which must be stationary. Hence for a certain $r$ we have for any $i$:
$NJ^i \cap MJ^{i-r}=NJ^i \cap MJ^{i-r-1}=\cdots=NJ^i\cap M$. Putting $i=n+r$ yields
$MJ^{n}\supset NJ^{n+r}\cap MJ^n=NJ^{n+r}\cap M$.
\end{proof}
There is a canonical functor
\begin{equation}
\label{ref-2.4-11}
\Phi:\Cscr\r \widehat{\Cscr}:M\mapsto \invlim_n(M/MJ^n)
\end{equation}
\begin{propositions}
 \label{ref-2.2.9-12}
The functor $\Phi$ introduced above is exact. It induces an equivalence
\[
\Cscr_{R/J^n}\cong (\widehat{\Cscr})_{R/J^n}
\]
\end{propositions}
\begin{proof}
Exactness  is a consequence of Lemma \ref{ref-2.2.8-10}. It is similar to the
  proof of exactness of $M\mapsto \invlim_n M/MJ^n$ in the proof of
  Proposition \ref{ref-2.2.7-9}.

The second statement is a tautology when written out formally. 
\end{proof}
\begin{lemmas}\label{ref-2.2.10-13} Let $\Phi:\Cscr\r \Dscr$ be a functor between $R$-linear noetherian abelian categories
which induces equivalences $\Phi_n:\Cscr_{R/J^n}\r \Dscr_{R/J^n}$. Then
\[
\widehat{\Phi}:\widehat{\Cscr}\r \widehat{\Dscr}:M\mapsto \invlim_n \Phi_n(M/MJ^n)
\]
is an equivalence. 
\end{lemmas}
\begin{proof} One easily checks that $\widehat{\Phi}$ is well defined and that its 
inverse is given by $\widehat{\Phi}^{-1}(N)=\invlim_n \Phi_n^{-1}(N/NJ^n)$. 
\end{proof}
\begin{definitions} A noetherian $R$-linear abelian category $\Cscr$ is \emph{complete}
if the functor $\Phi:\Cscr\r \widehat{\Cscr}$ is an equivalence.
\end{definitions}
\begin{propositions} $\widehat{\Cscr}$ is complete.
\end{propositions}
\begin{proof} This follows from the second statement of Proposition \ref{ref-2.2.9-12} combined
with Lemma \ref{ref-2.2.10-13}. 
\end{proof}
For completeness let us recall the following result
\begin{lemmas} \label{ref-2.2.13-14} (Nakayama) Let $M\in \widehat{\Cscr}$ be
    such that $MJ=M$. Then $M=0$.
\end{lemmas}
\begin{proof} We have $M=\invlim_n M/MJ^n=0$ since $M=MJ=MJ^2=\cdots$. 
\end{proof}
\subsection{Functors}
Now we consider functors. Let $T=(T^i)_i$ be a $\partial$-functor
between~$R$-linear noetherian abelian categories $\Cscr$ and
$\Dscr$. We extend $T$ to a $\partial $-functor~$\hat{T}$ commuting
with filtered inverse limits between $\Pro(\Cscr)$ and $\Pro(\Dscr)$.

 In this section we prove the following strengthening of \cite[Prop.\ 5.3.1]{Joua}.
\begin{theorems} \label{ref-2.3.1-15} The functor $\hat{T}$ sends $\widehat{\Cscr}$ to $\widehat{\Dscr}$. 
\end{theorems}
\begin{proof} This is a variant of \cite[Prop.\ 5.3.1]{Joua}. For the convenience
of the reader we adapt the proof in loc.\ cit.\ to our setting.

We need some rudiments from the foundation of the theory of spectral sequences. 
In its abstract form a spectral sequence over an abelian category $\Escr$ is a
sequence of complexes $\mathbb{E}=(E^\bullet_r,d_r)_{r\ge 1}$ together with isomorphisms
$H^\bullet(E_r^\bullet,d)\cong (E_{r+1}^\bullet,0)$. If the terms of the complexes
$E^\bullet_r$ carry a grading then we assume that $d_r$ is homogeneous. Note that this
setup is shifted with respect to the usual indexing of
spectral sequences. This is more convenient for filtered objects. 

Starting from spectral sequence $\mathbb{E}$ we may construct subobjects
\[
0=B_1^n\subset\cdots \subset B^n_r\subset\cdots \subset Z_r^n\subset \cdots \subset Z_1^n=E^n_1
\]
with $E^n_{r}=Z_r^n/B_r^n$. The subobjects $B^n_r,Z^n_r$ are constructed 
recursively using the following exact sequences. 
\begin{equation}
\label{ref-2.5-16}
0\r Z_{r+1}^n/B_r^n\r Z_r^n/B_r^n\xrightarrow{d_r^n} Z_r^{n+1}/B_r^{n+1}\r
Z_{r}^{n+1}/B_{r+1}^{n+1}\r 0
\end{equation}
If $Z_\infty^n=\invlim_r Z^n_r$, $B_\infty^n=\dirlim_r B^n_r$  exist then
we say that $\mathbb{E}$ \emph{converges} to $E_\infty^n=Z^n_\infty/B^n_\infty$. The graded object
$E^\bullet_\infty$ is called the limit of the spectral sequence. 

The spectral sequence is said to \emph{degenerate} at $E^n_{r_0}$ if $d^n_r=0$, $d^{n-1}_{r}=0$
for $r\ge r_0$.
In that case it follows from \eqref{ref-2.5-16} that $B^n_{r+1}=B^n_{r}$, $Z^n_{r+1}=Z^n_{r}$ for $r\ge r_0$ and thus
$E_\infty^n$ exists and is equal to $E_{r_0}^n$.

If $(T^n)_n:\Escr\r \Fscr$ is a $\partial$-functor between abelian
categories and $X\in\Escr$ is an object equipped with a descending
filtration $F_{k+1}X\subset F_kX$ indexed by $\ZZ$ then the method of
exact couples yields a spectral sequence starting with
$E^n_1=T^n(\gr_F X)$. Here $ T^n(\gr_F
X)\overset{\text{def}}{=}\bigoplus_k T^n(F_kX/F_{k+1}X)$ is viewed as
a $\ZZ$-graded object over $\Fscr$ (through the $k$-index).  Hence
this spectral sequence lives in the abelian category of $\ZZ$-graded objects
over $\Fscr$. The expressions for $Z^n_r$ and $B^n_r$ are
\begin{align*}
Z^n_r&=\bigoplus_k\ker(T^n(F_kX/F_{k+1}X)\r T^{n+1}(F_{k+1}X/F_{k+r}X))\\
B^n_r&=\bigoplus_k \im(T^{n-1}(F_{k-r+1}X/F_{k}X)\r T^n(F_kX/F_{k+1}X))
\end{align*}
We now make a number of hypotheses.
\begin{enumerate}
\item $\Escr$, $\Fscr$ are complete with exact filtered limits. 
\item $(T^n)_n$ commutes with filtered limits.
\item We have $X=F_kX$ for $k\ll 0$. 
\item $X$ is complete. I.e.\ $X=\invlim_k X/F_kX$. 
\end{enumerate}
We note that limits and colimits on graded objects can be computed degreewise. Hence
$Z_\infty^n$ exists and is equal to
\[
Z^n_\infty=\bigoplus_k\ker(T^n(F_kX/F_{k+1}X)\r T^{n+1}(F_{k+1}X))\\
\]
Similarly $B_\infty^n$ exists and is equal to
\[
B^n_\infty=\bigoplus_k \im(T^{n-1}(X/F_kX)\r T^n(F_kX/F_{k+1}X))
\]
It is now well-known and an easy verification that
\[
Z^n_\infty/B^n_\infty=\bigoplus_k \frac{\im(T^n(F_k X)\r T^nX)}{\im(T^n(F_{k+1} X)\r T^n X)}
\]
In other words if we equip $T^nX$ with the filtration $F_k(T^nX)=\im (T^n(F_kX)\r T^nX)$
then $E_\infty^n\cong \gr_F T^n X$. Note that the conditions also imply
that $T^nX$ is complete for this filtration. Indeed 
\begin{align*}
\invlim_k T^nX/F_k T^n X&=\invlim_k\coker (T^n(F_kX)\r T^nX)\\
&=\coker( T^n(\invlim_k F_kX)\r T^nX)\\
&=T^nX
\end{align*}
Now revert to the notations in the statement of the proposition. We
apply the previous discussion with $\Escr=\Pro(\Cscr)$,
$\Fscr=\Pro(\Dscr)$ and $X=M$.  We equip $M$ with the $J$-adic
filtration. By the above discussion we get a spectral sequence
$\mathbb{E}$ with $E^n_1=T^n(\gr_{F_J} M)$ which converges to $\gr_F
\hat{T}^n(M)$.  The terms occuring in this spectral sequence are graded
$(\gr_J R,\Dscr_{R/J})$ objects. The limit is a priori only a $(\gr_J
R,(\Pro\Dscr)_{R/J})$-object.

By Lemma \ref{ref-2.2.2-3} $\gr_J M$ is a noetherian $(\gr_J R,\Cscr_{R/J})$ object.
Hence by Lemma \ref{ref-2.3.2-17} below $T^n(\gr_J M)$ is a noetherian graded $(\gr_J R,\Dscr_{R/J})$-object.
Hence the ascending chain $B_r^n$ must be stationary. By \eqref{ref-2.5-16}
we obtain $d^n_r=0$ for $r\gg 0$. Hence~$\mathbb{E}$ degenerates at $E^n_r$ for $r\gg 0$.
It follows that $E_\infty^n=\gr_F \hat{T}^n(M)$ is a noetherian graded $(\gr_J R,\Dscr_{R/J})$-object.

Since we had already shown that $\hat{T}^n(M)$
is complete we conclude by Lemma \ref{ref-2.2.3-4}.
\end{proof}
\begin{lemmas} \label{ref-2.3.2-17} Let $S_0$ be a commutative noetherian ring and let $S$ be a
  finitely generated positively commutative graded $S_0$-algebra whose part of
  degree zero is $S_0$.  Let $(T^i)_i$ be a $\partial$ functor between
  noetherian abelian categories $S_0$-linear categories $\Escr$,
  $\Fscr$. Then for any noetherian graded $(S,\Escr)$-object $N$ and
  for any $i$ we have that $T^i(N)\overset{\text{def}}{=}\bigoplus_n
  T^i(N_n)$ is a noetherian graded $(S,\Fscr)$-object.
\end{lemmas}
\begin{proof} We perform induction on the minimal number of generators
  $d$ of $S$ as $S_0$-algebra.  If $d=0$ then $S=S_0$ and hence $N$ is
  concentrated in a finite number of degrees. In this case $T^i(N)$ is
  obviously noetherian.

Now assume $d>0$ and pick a homogeneous generator $t$ of $S$ over $S_0$ of strictly
positive degree $f$. Then $N$ can be written as an extension
\[
0\r N'\r N\r N''\r 0
\]
 where $N'$ is annihilated by some power of $t$ and  $N''$ is $t$-torsion free.
The object $N'$ can itself be written as a repeated extension of objects annihilated by $t$.
Thus it suffices to treat the cases where $N$ is annihilated by $t$ and where $N$ is $t$-torsion free. 

If $N$ is annihilated by $t$ then $T^i(N)$ is noetherian by induction (since $N$ is 
now an $S/tS$-module and $S/tS$ has one generator less than $S$). Hence we assume
that $N$ is $t$-torsion free. From the exact sequence
\[
0\r N\xrightarrow{\times t} N\r N/Nt\r 0
\]
we obtain an injection
\[
T^i(N)/T^i(N)t\hookrightarrow T^i(N/Nt)
\]
By induction $T^i(N/Nt)$ is noetherian and hence so is $T^i(N)/T^i(N)t$. From
this one easily obtains that $T^i(N)$ itself is noetherian. 
\end{proof}
\subsection{Formal flatness}
\label{ref-2.4-18}
The notations $R,J,\Cscr$ are as above. For use below it would be convenient to assume 
that $\Cscr$ if flat. Unfortunately even if $\Cscr$ is $R$-flat then
there seems to be no a priori reason for $\widehat{\Cscr}$ to be flat
(although we do not know an explicit counter example). 
To work around this issue we make the following definition
\begin{definitions} The $R$-linear category $\Cscr$ is \emph{formally flat} if the categories $\Cscr_{R/J^n}$
are $R/J^n$-flat for all $n$.
\end{definitions}
Since $\widehat{\Cscr}_{R/J^n}=\Cscr_{R/J^n}$ it immediately follows that if $\Cscr$ is
formally flat then so is~$\widehat{\Cscr}$. 
The following proposition yields a different characterization of formal flatness.
\begin{propositions} \label{ref-2.4.2-19} Let $\Cscr_t$ be the full subcategory of objects
  in $\Cscr$ that are annihilated by some power of $J$. This is
  naturally an $R$-linear category. Then $\Cscr$ is formally flat if
  and only if $\Cscr_t$ is flat.
\end{propositions}
\begin{proof}
Assume first that $\Cscr$ is formally flat. Take $M\in \Cscr_t$, $N\in \mod(R)$. We must
prove that $\Tor_i^R(M,N)$ is effaceable in its first argument in $\Cscr_t$ for $i>0$. We
assume $M\in \Cscr_{R/J^n}$.

By dimension shifting in $N$ we may reduce to the case $i=1$. Take an exact sequence in $\mod(R)$
\[
0\r N'\r P\r N\r 0
\]
with $P$ projective. Then we get an exact sequence
\begin{equation}
\label{ref-2.6-20}
0\r \Tor_1^R(M,N)\r M\otimes_R N'\r M\otimes_R P\r M\otimes_R N\r 0
\end{equation}
Let $F$ be the filtration on $N'$ induced from the $J$-adic filtration on $P$. 
From
\[
0\r N'/F_lN' \r P/PJ^l\r N/NJ^l\r 0
\]
we obtain an exact sequence (for $l\ge n$)
\[
0\r \Tor_1^{R/J^l}(M,N/NJ^l)\r M\otimes_{R/J^l} N'/F_lN'\r M\otimes_{R/J^l} P/PJ^l\r M\otimes_{R/J^l} N/NJ^l\r 0
\]
Combining these two sequence we see that there is a map
\begin{equation}
\label{new-2}
\Tor_1^R(M,N)\r \Tor_1^{R/J^l}(M,N/NJ^l)
\end{equation}
natural in $M$ (for $l\ge n$).

By the
Artin-Rees condition (Lemma \ref{ref-2.2.8-10}) we may take an $l$ such that $F_lN'\subset N'J^n$.
Then from \eqref{ref-2.6-20} we obtain an exact sequence
\[
0\r \Tor_1^R(M,N)\r M\otimes_{R/J^l} N'/F_lN'\r M\otimes_{R/J^l} P/PJ^l\r M\otimes_{R/J^l} N/NJ^l\r 0
\]
and thus we have deduced that for $l$ large \eqref{new-2} is an isomorphism
\begin{equation}
\label{ref-2.7-21}
\Tor_1^R(M,N)= \Tor_1^{R/J^l}(M,N/NJ^l)
\end{equation}
Since $\Cscr_{R/J^l}$ is flat $ \Tor_1^{R/J^l}(M,N/NJ^l)$ is
effaceable in its first argument in $\Cscr_{R/J^l}$ by an epimorphism
$M'\r M$. Thus we get a commutative diagram
\[
\xymatrix{
\Tor_1^{R}(M',N)\ar[r]\ar[d]&\Tor_1^{R/J^l}(M',N/NJ^l)\ar[d]\\
\Tor_1^{R}(M,N)\ar@{=}[r]&\Tor_1^{R/J^l}(M,N/NJ^l)\\
}
\]
with the right most map being zero. It follows that the left most map
is also zero.  Thus $\Tor_1^R(M,N)$ is effaceable in $M$ in $\Cscr_t$ and
hence $\Cscr_t$ is flat. 

Conversely assume $\Cscr_t$ is flat. Since $\Cscr_{t,R/J^n}=\Cscr_{R/J^n}$ and flatness is
stable under base change (\cite[Prop.\ 4.8]{lowenvdb1}) we conclude that $\Cscr_{R/J^n}$ is flat. 
\end{proof}
By Lemma  \ref{ref-2.2.10-13} the inclusion of abelian categories
$
\Cscr_t \r\Cscr
$
yields an equivalence
\[
\widehat{\Cscr}_t\cong \widehat{\Cscr}
\]
Hence when $\Cscr$ is formally flat we may always reduce to the case that $\Cscr$ is flat.

\medskip

Flatness on the level of objects does not present any pitfalls as the following proposition
shows.
\begin{propositions} Let $M\in\widehat{\Cscr}$. Then $M$ is $R$-flat if and only if $M/MJ^n$ is
$R/J^n$-flat for all $n$.
\end{propositions}
\begin{proof} We consider the non-obvious direction. Assume that $M\in\widehat{\Cscr}$ is
such that all $M/MJ^n$ are flat. We need to prove that $M\otimes_R -$ is exact. 

Consider
an exact sequence in $\mod(R)$. 
\[
0\r K\r L\r N\r 0
\]
We have to show that 
\[
0\r M\otimes_R K\r M\otimes_R L\r M\otimes_R N\r 0
\]
is exact. 
After tensoring with $R/J^n$ is is sufficient to show that 
\[
0\r M\otimes_R K/KJ^n\r M\otimes_R L/LJ^n\r M\otimes_R N/NJ^n\r 0
\]
is exact up to essentially zero systems. This is the same sequence as
\[
0\r M/MJ^n\otimes_{R/J^n} K/KJ^n\r M/MJ^n\otimes_{R/J^n} L/LJ^n\r M/MJ^n\otimes_{R/J^n} N/NJ^n\r 0
\]
Hence by flatness of $M/MJ^n$ it is sufficient that 
\[
0\r K/KJ^n\r L/LJ^n\r N/NJ^n\r 0
\]
is exact up to essentially zero systems. This follows from the
Artin-Rees condition (see Lemma \ref{ref-2.2.8-10}).
\end{proof}
\subsection{$\Ext$-groups}
Now we discuss $\Ext$-groups.
 by which we always mean Yoneda $\Ext$-groups.
We keep the notations from the previous section. Thus $(R,J)$ is a
noetherian $J$-adic ring and $\Cscr$ is an $R$-linear noetherian
abelian category.

We make the following definition
\begin{definitions} Let $M,N\in \widehat{\Cscr}$. 
Then the \emph{completed $\Ext$-groups} between $M$,$N$ are defined by
\begin{align*}
`\Ext_{\widehat{\Cscr}}^i(M,N)&=\Ext_{\Pro(\Cscr_t)}^i(M,N)
\end{align*}
\end{definitions}
It is clear that $`\Ext_{\widehat{\Cscr}}^i(M,N)$ is a $\partial$-functor
in both arguments. Apart from this nice property we don't know
if completed $\Ext$-groups are meaningful objects in general. To get better control we will assume
that $\Cscr$ is formally flat 
and we impose an additional  finiteness condition
\begin{definitions} An $R$-linear abelian category $\Dscr$ is \emph{$\Ext$-finite}
if for all objects $M,N\in \Dscr$ we have that $\Ext^i_\Dscr(M,N)$ is a finitely
generated $R$-module for all $i$. 
\end{definitions}
The following will be the main result of this section
\begin{propositions} \label{ref-2.5.3-22} Assume that $\Cscr$ is formally flat and that
  $\Cscr_{R/J}$ is $\Ext$-finite.  Then  $`\Ext_{\widehat{\Cscr}}^i(M,N)\in \mod(\hat{R})$  for $M,N\in
  \widehat{\Cscr}$
and furthermore
\begin{align}
`\Ext_{\widehat{\Cscr}}^i(M,N)
&=\invlim_l \Ext_{\widehat{\Cscr}}(M,N/NJ^k)\label{new-4}\\
&=\invlim_k \dirlim_l \Ext_{\Cscr_{R/J^l}}^i(M/MJ^l,N/NJ^k) \label{new-5}
\end{align}
If $M$ is in addition $R$-flat then
\begin{equation}
\label{new-6}
`\Ext_{\widehat{\Cscr}}^i(M,N)=\invlim_k  \Ext_{\Cscr_{R/J^k}}^i(M/MJ^k,N/NJ^k)
\end{equation}
\end{propositions}
The proof is a series of lemmas.
\begin{lemmas} \label{ref-2.5.4-23} Assume that $\Cscr$ is formally flat and that $\Cscr_{R/J}$ is $R/J$-$\Ext$ finite then $\Cscr_{R/J^n}$ is
$R/J^n$-$\Ext$-finite for all $n$. 
\end{lemmas}
\begin{proof} We need to prove that for $M,N\in \Cscr_{R/J^n}$ we have
that $\Ext^i_{\Cscr_{R/J^n}}(M,N)$ is a finitely generated $R/J^n$-module. By 
filtering $M$ we may assume $M\in \Cscr_{R/J}$. Then we conclude by 
a change of rings spectral sequence (which depends on flatness, see \cite[Prop.\ 4.7]{lowenvdb1})
\[
E^{pq}_2:\Ext^p_{\Cscr_{R/J}}(M,\Ext^q_{\Cscr_{R/J^n}}(R/J,N))\Rightarrow
\Ext^{p+q}_{\Cscr_{R/J^n}}(M,N)\qed
\]
\def\qed{}\end{proof}
The following lemma is proved in a similar way.
\begin{lemmas} \label{ref-2.5.5-24} Assume that $\Cscr$ is formally flat and that $\Cscr_{R/J}$  is $\Ext$-finite. Let $M\in \widehat{\Cscr}$
and $N\in \Cscr_{R/J^n}$. Then $\Ext_{\Pro(\Cscr_t)}^i(M,N)$ is a finitely generated
$R/J^n$-module for all $i$.  Furthermore we have
\begin{equation}
\label{ref-2.8-25}
\Ext_{\Pro(\Cscr_t)}^i(M,N)=\dirlim_l \Ext_{\Cscr_{R/J^l}}^i(M/MJ^l,N)
\end{equation}
and
\begin{equation}
\label{ref-2.9-26}
\Ext_{\Pro(\Cscr_t)}^i(M,N)=\Ext_{\widehat{\Cscr}}^i(M,N)
\end{equation}
Finally if $M$ is flat over $R$ then
\begin{equation} \label{ref-2.10-27}
\Ext_{\Pro(\Cscr_t)}^i(M,N)=\Ext_{\Cscr_{R/J^n}}^i(M/MJ^n,N)
\end{equation}
\end{lemmas}
\begin{proof} If $\Cscr_t$ is flat then so is $\Pro(\Cscr_t)$ (see \cite[Prop.\ 3.6]{lowenvdb1} 
for the dual statement). 
 Now we use the spectral sequence (for $l\ge n$)
\begin{equation}
\label{ref-2.11-28}
E^{pq}_2(l):\Ext_{\Pro(\Cscr_t)_{R/J^l}}^p(\Tor^R_q(M,R/J^l),N)\Rightarrow \Ext_{\Pro(\Cscr_t)}^{p+q}(M,N)
\end{equation}
which may derived in a similar way as \cite[Prop.\ 4.7]{lowenvdb1} (the existence
depends on flatness of $\Pro(\Cscr_t)$).  The formation
of $\Pro$-objects commutes with certain base changes (see \cite[Prop.\ 4.5]{lowenvdb1} for
the dual statement) 
and in particular $\Pro(\Cscr_t)_{R/J^l}=\Pro(\Cscr_{R/J^l})$.

Since
$\Tor^R_q(M,R/J^l)$ lies both in $\widehat{\Cscr}$ and is annihilated by $J^l$ it
lies in $\Cscr_{R/J^l}$. For an object $K\in \Cscr_{R/J^l}$ we have $\Ext^i_{\Pro(\Cscr_{R/J^l})}(
K,N)=\Ext^i_{\Cscr_{R/J^l}}(
K,N)$ (see \cite[Prop.\ 2.14]{lowenvdb1}). Hence the spectral sequence \eqref{ref-2.11-28} becomes
\begin{equation}
\label{ref-2.12-29}
E^{pq}_2(l):\Ext_{\Cscr_{R/J^l}}^p(\Tor^R_q(M,R/J^l),N)\Rightarrow \Ext_{\Pro(\Cscr_t)}^{p+q}(M,N)
\end{equation}
To prove finite generation we put $l=n$ and  invoke  Lemma \ref{ref-2.5.4-23}.

To prove \eqref{ref-2.8-25} we note that for $l\le l'$ there are maps
of spectral sequence $E(l)\r E(l')$ which are given by the compositions
\[
\Ext_{\Cscr_{R/J^l}}^p(\Tor^R_q(M,R/J^l),N)\r \Ext_{\Cscr_{R/J^{l'}}}^p(\Tor^R_q(M,R/J^l),N)
\r \Ext_{\Cscr_{R/J^{l'}}}^p(\Tor^R_q(M,R/J^{l'}),N)
\]
It follows from Lemma \ref{ref-2.5.6-30} below that $E^{pq}_2(l)\r
E^{pq}_2(l')$ is zero for $q>0$ and $l\ll l'$ (taking into account that $M$ is a noetherian 
object in $\widehat{\Cscr}$). Taking a direct limit over $l$ of \eqref{ref-2.12-29} we find that indeed
\[
\dirlim_l \Ext_{\Cscr_{R/J^l}}^p(M/MJ^l,N)= \Ext_{\Pro(\Cscr_t)}^{p+q}(M,N)
\]
The claim \eqref{ref-2.9-26} follows from Lemma \ref{ref-2.5.7-31} below together with \eqref{ref-2.8-25}. The claim \eqref{ref-2.10-27} follows from the degeneration of the spectral
sequence \eqref{ref-2.12-29}.
\end{proof}
We have used the next two lemmas.
\begin{lemmas} \label{ref-2.5.6-30}
Let $\Dscr$ be an $R$-linear abelian category and assume
  that $M\in \Dscr$ is a noetherian object. Then $\Tor_i^R(M,R/J^n)_n$ is an
essentially zero system for $i>0$.
\end{lemmas}
\begin{proof} We first replace $\Dscr$ with the smallest abelian subcategory of $\Dscr$
containing~$M$. This is a noetherian abelian category. Then we replace $\Dscr$ by its
category of $\Ind$-objects. Then $\Dscr$ becomes a locally noetherian Grothendieck category 
and $M$ is still a noetherian object in $\Dscr$. In particular $\Ext^i_\Dscr(M,-)$ commutes
with filtered colimits. 

As a Grothendieck category $\Dscr$ has enough injectives. Hence we have to show that
for an arbitrary injective object $E\in \Dscr$ we have
\[
\dirlim_n \Hom_\Dscr(\Tor_i^R(M,R/J^n),E)=0
\]
Now we have
\begin{align*}
\dirlim_n \Hom_\Dscr(\Tor_i^R(M,R/J^n),E)&=\dirlim_n \Ext^i_\Dscr(M,\Hom_R(R/J^n,E))\\
&=\Ext^i_\Dscr(M,\dirlim_n \Hom_R(R/J^n,E))
\end{align*}
Thus we have to show that  $F=\dirlim_n \Hom_R(R/J^n,E)$ is injective. In a locally
noetherian Grothendieck category we can test this on inclusions of noetherian
objects. Hence let $K\hookrightarrow M$ be such an inclusion. We need to prove that
$\dirlim_n \Hom(M/MJ^n,E)\r \dirlim_n\Hom(K/KJ^n,E)$ is an epimorphism, or equivalently
that the kernel of $K/KJ^n \r M/MJ^n$ is an essentially zero system. This follows from
the Artin-Rees condition (Lemma \ref{ref-2.2.8-10}).
\end{proof}
\begin{lemmas} \label{ref-2.5.7-31} Let $\Dscr$ be a noetherian $R$-linear category and assume that $N\in \Dscr$
is annihilated by $J^n$. Let $M\in \Dscr$. Then we have
\[
\Ext^i_{\Dscr}(M,N)=\dirlim_l \Ext^i_{\Dscr_{R/J^l}}(M/MJ^l,N)
\]
\end{lemmas}
\begin{proof} 
From the Artin-Rees condition (Lemma \ref{ref-2.2.8-10}) we know that $(-\otimes_{R} R/J^l)_l$ is exact
up to essentially zero systems.  From this we obtain in the usual way
that $M\mapsto \dirlim_l \Ext^i_{\Dscr_{R/J^l}}(M/MJ^l,N)$ is a $\partial$-functor with
values in $R$-modules.
To show that this $\partial$-functor coincides with $\Ext^i_{\Dscr}(M,N)$ it is
sufficient to prove this for $i=0$ and to show that any element of
$\dirlim_l \Ext^i_{\Dscr_{R/J^l}}(M/MJ^l,N)$ is effaceable for $i>0$. The case $i=0$
is trivial so assume that $a\in \Ext^i_{\Dscr_{R/J^l}}(M/MJ^l,N)$ represents an element
$\bar{a}$ of  $\dirlim_l \Ext^i_{\Dscr_{R/J^l}}(M/MJ^l,N)$ for $i>0$. 

There exists an epimorphism $T\r M/MJ^l$ which effaces $a$ for some
$T\in \Dscr_{R/J^l}$.  Let $M'$ be the pullback of $T$ for the map
$M\r M/MJ^l$. Then the epimorphism $M'/M'J^l\r M/MJ^l$ factors through $T$ and
hence effaces $a$. This means that $M'\r M$ effaces $\bar{a}$
 \def\qed{}\end{proof}
 \begin{lemmas} \label{ref-2.5.8-32} Assume that $\Cscr$ is formally flat and that $\Cscr_{R/J}$ is $\Ext$-finite and 
 \label{ref-2.5.8-33} 
 let $M,N\in \widehat{\Cscr}$. Then 
 \[
 \invlim_n \Ext^i_{\Pro(\Cscr_t)}(M,N/NJ^n)\in \mod(R)\,\hat{}\quad(\cong \mod(\hat{R}))
 \]
 \end{lemmas}
 \begin{proof} This follows by applying Theorem \ref{ref-2.3.1-15} to the $\partial$-functor
 \[
 \Ext^i_{\Pro(\Cscr_t)}(M,-)_i:\Cscr_t\r \mod(R)_t
 \]
  To construct this functor we use Lemma \ref{ref-2.5.5-24}. 
 \end{proof}
 \begin{proof}[Proof of Proposition \ref{ref-2.5.3-22}]
 Let $P_\bullet$ be a projective resolution of
 $M$ in $\Pro(\Cscr_t)$. Since $\Cscr_t$ is flat by Proposition \ref{ref-2.4.2-19} the $P_m$ are
 $R$-flat (see \cite[Prop.\ 3.4]{lowenvdb1} for the dual version).  We compute
 \begin{align*}
 \Ext^i_{\Pro(\Cscr_t)}(M,N)&=H^i(\Hom_{\Pro(\Cscr_t)}(P_\bullet,\invlim_n N/NJ^n))\\
 &=H^i(\invlim_n \Hom_{\Pro(\Cscr_t)}(P_\bullet, N/NJ^n))
 \end{align*}
 We need to exchange $H^i$ and $\invlim_n$. This is possible if the terms of the
 inverse system of complexes $\Hom_{\Pro(\Cscr_t)}(P_\bullet, N/NJ^n)$ as well as its cohomology
 satisfy the Mittag-Leffler condition. For the terms this follows from the projectivity of
 $P_m$. For the cohomology which is equal to $\Ext^i_{\Pro(\Cscr_t)}(M,N/NJ^n)$ we 
 invoke Lemma \ref{ref-2.5.8-32} together with Lemma \ref{ref-2.5.9-34} below. 

 Assuming this we now obtain
 \begin{align*}
 \Ext^i_{\Pro(\Cscr_t)}(M,N)&=\invlim_n H^i( \Hom_{\Pro(\Cscr_t)}(P_\bullet, N/NJ^n))\\
 &=\invlim_n\Ext^i_{\Pro(\Cscr_t)}(M,N/NJ^n)
 \end{align*}
This implies \eqref{new-4} via \eqref{ref-2.9-26}. Furthermore we obtain 
\eqref{new-5} via \eqref{ref-2.10-27}. Finally we obtain \eqref{new-6} from \eqref{ref-2.9-26}.
\end{proof}
The following lemma was used. 
\begin{lemmas} 
\label{ref-2.5.9-34} Let $(U_n)_n$ be an inverse system in an $R$-linear noetherian
abelian category $\Dscr$ such
  that $U_nJ^n=0$ and such that the pro-object $(U_n)_n$ lies in
  $\widehat{\Dscr}$. Then $(U_n)_n$ satisfies the Mittag-Leffler
  condition.
\end{lemmas}
\begin{proof} Let $U$ be the pro-object in $(U_n)_n$. Define $C_n$,
  $K_n\in \Dscr$ as the kernel and cokernel of the natural maps.
\[
U/UJ^n\r U_n
\]
Taking inverse limits in $\Pro(\mod(R))$ we see that $(U_n)_n$ and
$(C_n)_n$ are zero pro-objects, or equivalently they are essentially
zero systems. From one easily deduces that $(U_n)_n$ satisfies
the Mittag-Leffler condition.
\end{proof}
\subsection{The complete derived category}
We use the same notations as above. In particular $(R,J)$ is a $J$-adic noetherian
ring and $\Cscr$ is an $R$-linear noetherian abelian category. We define $D_c(\Cscr)$ as
the full subcategory of $D(\Pro(\Cscr_t))$ of complexes whose cohomology lies 
in $\widehat{\Cscr}$. Thus $D_c(\Cscr)$ has a $t$-structure whose heart is
$\widehat{\Cscr}$. 
Then the completed $\Ext$-groups for $M,N\in \widehat{\Cscr}$ may be reinterpreted as
\[
{}`\Ext_{\widehat{\Cscr}}(M,N)=\Hom_{D_c(\Cscr)}(M,N[n])
\]
In the case that $\Cscr$ is the category of torsion $l$-adic constructible
sheaves it would be interesting to compare this derived category
to the standard derived category of $l$-adic sheaves \cite{Behrend1,BBD,Ek}.

\section{Formal deformations of abelian categories}
\label{ref-3-35}
\def\Def{\operatorname{Def}}
Let $R$ be a noetherian $J$-adic ring and $\Cscr$ a flat $R/J$-linear noetherian abelian category. Then we
define an $R$-deformation of $\Cscr$ to be a formally flat complete $R$-linear abelian category $\Dscr$
together with an equivalence $\Dscr_{R/J}\cong \Cscr$. It follows from the above discussion
that $\Dscr$ is specified up to equivalence by specifying the flat 
$R/J^n$-linear categories $\Dscr_n=\Dscr_{R/J^n}$ together with the equivalences
(isomorphisms in this case) $\Dscr_m=\Dscr_{n,R/J^m}$ for $n\ge m$ and $\Dscr_1\cong \Cscr$. 

Denote by $\Def_R(\Cscr)$ the class of $R$-deformations of $\Cscr$. This is a 2-groupoid. 
The observations in the previous paragraph may be used to construct a 2-equivalence
\[
\Def_R(\Cscr)\cong 3{\text{-}}\invlim_n \Def_{R/J^n}(\Cscr)
\]
We leave it to the interested reader to formalize this statement. It
will not be used in this form.
\section{Ampleness}
\label{ref-4-36}
We define what we mean by a strongly ample sequence. This is stronger than
strictly necessary but easier to work with.

Let $\Escr$ an noetherian abelian category. For us a sequence $(O(n))_{n\in
  \ZZ}$ of objects in $\Escr$ is strongly ample if the following conditions
hold
\begin{itemize}
\item[(A1)] For all  $M\in \Escr$ and for all $n$
there is an epimorphism $\oplus_{i=1}^t O(-n_i)\r M$ with $n_i\ge n$.
\item[(A2)] For all $M\in \Escr$ and for all $i>0$ one has
  $\Ext^i_\Escr(O(-n),M)=0$ for $n\gg 0$.
\end{itemize}
A strongly ample sequence $(O(n))_{n\in \ZZ}$ in $\Escr$ is ample in
the sense of \cite{Polishchuk1}. Hence using the methods of \cite{AZ} or \cite{Polishchuk1} one
obtains 
$
\Escr\cong \qgr(A)
$
if $\Escr$ is $\Hom$-finite, where $A$ is the noetherian $\ZZ$-algebra $\oplus_{ij} \Hom_\Escr(O(-j),O(-i))$.

\medskip

Below we fix a complete noetherian $J$-adic ring $R$. The following is a version
of Grothendieck's existence theorem \cite{EGA31}.
\begin{proposition}
\label{ref-4.1-37}
Let $\Escr$ be an $\Ext$-finite $R$-linear noetherian category with a
strongly ample sequence $(O(n))_n$. Then $\Escr$ is complete and
furthermore if $\Escr$ is flat we have for $M,N\in\Escr$:
\begin{equation}
\label{ref-4.1-38}
\Ext^i_\Escr(M,N)=`\Ext^i_\Escr(M,N)
\end{equation}
\end{proposition}
\begin{proof}
\begin{step}
\label{ref-1-39}
We first claim that $\Escr$ satisfies Nakayama's lemma. This would
follow from Lemma \ref{ref-2.2.13-14} once we knew $\Escr$ is complete but we are not there
yet.

Let $M\in\Escr$ me such that $MJ=M$. Pick generators for $a_1,\ldots,a_n$ for $J$
and consider the corresponding epimorphisms
\[
M^{\oplus n}\xrightarrow{(a_i)_i}M\r 0
\]
Applying $\Hom_\Escr(O(-m),-)$ for $m$ large we get an epimorphism
\[
\Hom_{\Escr}(O(-m),M)^{\oplus n}\xrightarrow{(a_i)_i}\Hom_{\Escr}(O(-m),M)\r 0
\]
and thus by Nakayama's lemma for $R$ and $\Ext$-finiteness
\[
\Hom_{\Escr}(O(-m),M)=0
\]
for large $m$. It then follows from (A1) that $M=0$. 
\end{step}
\begin{step}
\label{ref-2-40} Let $M$ be an object in $\Escr$. We claim that for $i>0$ and for 
large $m$ we have
\begin{equation}
\label{ref-4.2-41}
\Ext^i_\Escr(O(-m),MJ^n)=0
\end{equation}
for all $n$ and furthermore for large $m$ we also have
\begin{equation}
\label{ref-4.3-42}
\Hom_\Escr(O(-m),MJ^n)=\Hom_\Escr(O(-m),M)J^n
\end{equation}
for all $n$. 

Put $\tilde{R}=R\oplus J\oplus J^2\cdots$. According to Lemma \ref{ref-2.2.2-3}
$\tilde{M}=M\oplus MJ\oplus MJ^2\cdots$ is a noetherian graded $(\tilde{R},\Escr)$-object.

Now we follow a similar inductive method as in the proof of Lemma
\ref{ref-2.3.2-17}. We only give a sketch.  Put $W=\tilde{M}$. We need to prove
$\Ext^i_{\Escr}(O(-m),W)=0$ for large $m$. We deduce this from
the noetherian property of $W$.  Let $d$ be the minimal number of
generators of $\tilde{R}$ over $R$ (i.e.\ the number of generators
of the $R$-ideal~$J$). If $d=0$ there is nothing to prove. If $d>0$
then we pick a homogeneous generator~$t$ of strictly positive degree in $\tilde{R}$ and we
reduce to the cases $Wt=0$ and $W$ is $t$-torsion free.  The first
case is clear by induction and in the second case we obtain (also by induction)
$\Ext^i_{\Escr}(O(-m),W)t=\Ext^i_{\Escr}(O(-m),W)$ for large $m$.  The
conclusion now follows from the fact that $t$ has strictly positive degree. 

To prove \eqref{ref-4.3-42} we have to prove that
$\Hom_\Escr(O(-m),W)$ is generated in degree zero for $m$ large. We
use induction based on the noetherian property of $W$ and the fact
that $W$ is generated in degree zero. If $d=0$ then $W$ must be
concentrated in degree zero and there is nothing to prove. If $d\neq
0$ then we pick $t$ as above. Using $\Ext$-vanishing (which is already
proved) we obtain for $m$ large an exact sequence
\[
\Hom_{\Escr}(O(-m),W)\xrightarrow{\times t} \Hom_{\Escr}(O(-m),W)\r \Hom_\Escr(O(-m),W/Wt)\r 0
\]
We now finish by induction and the fact that $t$ has strictly positive degree.
\end{step}
\begin{step} 
Now we show that the functor
\[
\Phi:\Escr\r \widehat{\Escr}
\]
is an equivalence.  If $M$, $N\in \Escr$ then since $\Phi(M)$ is the pro-object
$M/MJ^n$ and similarly for $N$ we deduce from \eqref{ref-2.1-2}
we have
\begin{align*}
\Hom_{\widehat{\Escr}}(\Phi(M),\Phi(N))&=\invlim_k \dirlim_l \Hom_{\Escr_{R/J^l}}(M/MJ^l,N/NJ^k)\\
&=\invlim_k \Hom_{\Escr}(M,N/NJ^k)
\end{align*}
Thus we have to show that the natural map
\begin{equation}
\label{ref-4.4-43}
\Hom_\Escr(M,N)\r \invlim_n \Hom_\Escr(M,N/NJ^n)
\end{equation}
is an isomorphism.  Using left exactness of $\invlim$ and (A1) we immediately
reduce to $M=\Oscr(-m)$ where $m$ may be chosen arbitrarily large.

We first observe that \eqref{ref-4.4-43}  must be a monomorphism. Suppose on the contrary
that there is an $f:M\r N$ such that $K=\im f\subset NJ^n$ for all $n$. Then
by Lemma \ref{ref-2.2.8-10} we have $K=KJ$. Hence $K=0$ by Step 1 which implies $f=0$. 

Now we prove that \eqref{ref-4.4-43} is an epimorphism when $M=O(-m)$ for
$m$ large. Suppose we are given a a compatible system of maps $f_n:M\r
N/NJ^n$. By \eqref{ref-4.2-41} the $f_n$ may be lifted to maps
$f'_n:M\r N$ such that $\im(f'_n-f'_{n+1})\subset NJ^n$. Using \eqref{ref-4.2-41}
this implies $f'_n-f'_{n+1}\in \Hom_\Escr(M,N)J^n$. Since $\Hom_\Escr(M,N)\in \mod(R)$ the
limit $f=\lim_n f'_n$ exists. It has the property that the image of $f$ in $\Hom_\Escr(M,N/NJ^n)$
is equal to $f_n$. This proves what we want. 
\end{step}
\begin{step} Finally we prove \eqref{ref-4.1-38}. To show
that $`\Ext=\Ext$ in $\Escr$ it is sufficient to show that
$`\Hom=\Hom$ and furthermore that $`\Ext$ is effaceable in its first
argument.
  The fact that $`\Hom=\Hom$ is the fact that  \eqref{ref-4.4-43} is an isomorphism.
 So let us show that $`\Ext$ is
effaceable.

Let $N$ be an object in $\Escr$. If $m$ is large and $i>0$ then it
follows from \eqref{ref-4.2-41} that $\Ext^i_{\Escr}(O(-m),N/NJ^n)=0$ for all $n$. 
Hence $`\Ext^i_{\Escr}(O(-m),N)=0$ by \eqref{new-4}. Now let $M\in \Escr$. To efface $`\Ext^i_\Escr(M,N)$
we take an epimorphism $\bigoplus_{i=1}^t O(-n_i)\r M$ with the $n_i$ sufficiently large. This finishes
the proof.\qed
\end{step}
\def\qed{}\end{proof}
Now we fix an
$\Ext$-finite flat $R/J$-linear noetherian abelian category $\Cscr$ and 
an $R$-deformation $\Dscr$ of $\Cscr$. 
\begin{theorem}
\label{ref-4.2-44}
Let $O(n)_n$ be a sequence of $R$-flat objects in
$\Dscr$. Then 
\begin{enumerate}
\item $(O(n)/O(n)J)_n$ is strongly ample in $\Cscr$ if and only if $O(n)_n$ is
  strongly ample in $\Dscr$.
\item If $(O(n)/O(n)J)_n$ is strongly ample then $\Dscr$ is flat (instead of just formally flat).
\item $\Dscr$ is $\Ext$-finite as $R$-linear category. 
\end{enumerate}
\end{theorem}
\begin{proof}
Property (2) follows immediately from (A1) and (1).
To prove (1) we note that since the $O(n)$ are flat we have
\begin{equation}
\label{ref-4.5-45}
\Ext^i_{\Dscr}(O(n),M)=\Ext^i_{\Cscr}(O(n)/O(n)J,M)
\end{equation}
for $M\in\Cscr$. From this 
 it is easy to see that if $O(n)_n$
is strongly ample in $\Cscr$ then $(O(n)/O(n)J)_n$ is strongly ample in
$\Cscr$. So our main task is to prove the converse. 

So assume that $(O(n)/O(n)J)_n$ is strongly ample. We will first show
that $O(n)_n$ is strongly ample if we replace $\Ext$ by $`\Ext$ in the
definition of strongly ample.

Let $M$ be a noetherian object in $\Dscr$. Put
$S=R/J\oplus J/J^2\oplus\cdots$. According to Lemma \ref{ref-2.2.2-3}
$N=M/MJ\oplus MJ/MJ^2\oplus\cdots$ is a noetherian graded $(S,\Cscr)$-object.

We claim that for $m$ large we have
\begin{equation}
\label{ref-4.6-46}
\Ext^i_{\Dscr}(O(-m),MJ^n/MJ^{n+1})=0 \qquad\text{for all $n$} 
\end{equation}
This is again proved using a similar inductive method as in the proof of 
Lemma \ref{ref-2.3.2-17} and Proposition \ref{ref-4.1-37} (Step 2). We leave the proof to the reader.

In particular we find that for $m$ large one has
\[
\Ext^i_{\Dscr}(O(-m),M/MJ^n)=0 \qquad\text{for all $n$}
\]
Taking inverse limits and using Proposition \ref{ref-2.5.3-22}
we find $`\Ext^i_{\Dscr}(O(-m),M)=0$ for all
$i>0$. Hence this proves (A2) for $`\Ext$.

We now prove (A1) (which does not involve any $\Ext$). We first find
$m_0$ such that \eqref{ref-4.6-46} holds for $i=1$ and $m\ge m_0$.
Using (A1) for $\Cscr$ we find that there is an epimorphism
$F\overset{\text{def}}{=}\bigoplus_{i=1}^tO(-n_i)\r M/ MJ$ with
$n_i\ge m$.  We may lift this map to a compatible series of maps $F\r
M/MJ^n$. Taking the inverse limit yields a map $F\r M$. By Nakayama's
lemma (see Lemma \ref{ref-2.2.13-14}) it follows that this must be an
epimorphism.

So we have proved ampleness with $`\Ext$ replacing $\Ext$. Now we
claim that in fact $`\Ext=\Ext$ in $\Dscr$. This is proved in the same
way as Step 4 of the proof of Proposition \ref{ref-4.1-37} (using condition (A1) which was
already proved). 

Property (3) follows from Proposition \ref{ref-2.5.3-22}.
\end{proof}
\section{Lifting and base change}
\label{seclb}
The usual lifting results for infinitesimal deformations  generalize
without difficulty to formal deformations.  As usual $R$ is a $J$-adic noetherian ring
and we assume that $\Dscr$ is an $R$-deformation of a noetherian $\Ext$-finite $R/J$-linear
flat abelian category $\Cscr$. To simplify the notation we assume $\Cscr=\Dscr_{R/J}$. 
\begin{proposition}
\label{ref-5.1-47}
  Let $M\in \Cscr$ be a flat object such that
  $\Ext^i_\Cscr(M,M\otimes_{R/J} J^n/J^{n+1})=0$ for $i=1,2$ and $n\ge 1$.
  Then there exists a unique  $R$-flat object (up to non-unique
  isomorphism) $\overline{M}\in \Dscr$  such that $\overline{M}/\overline{M}J\cong M$.
\end{proposition}
\begin{proof} This follows in a straightforward way from the infinitesimal lifting
criterion for objects (see \cite[Theorem A]{lowen4} for the dual version).
\end{proof}
\begin{proposition}
\label{ref-5.2-48}
Let $\overline{M},\overline{N}\in \Dscr$ be flat objects and put
$\overline{M}/\overline{M}J= M$, $\overline{N}/\overline{N}J= N$. Assume that for all $X$
in $\mod(R/J)$ we have $\Ext^i_\Cscr(M,N\otimes_{R/J} X)=0$ for a certain $i>0$. Then
we  have $`\Ext^i_\Dscr(\overline{M},\overline{N}\otimes_R X)=0$ for all $X\in \mod(R)$. 
\end{proposition}
\begin{proof} This follows in a straightforward way from \cite[Prop.\ 6.13]{lowenvdb1}.
\end{proof}
\begin{proposition}
\label{ref-5.3-49}
Let $\overline{M},\overline{N}\in \Dscr$ be flat objects and put
$\overline{M}/\overline{M}J= M$, $\overline{N}/\overline{N}J=
N$. Assume that for all $X$ in $\mod(R/J)$ we have
$\Ext^1_\Cscr(M,N\otimes_{R/J} X)=0$. Then $\Hom_\Dscr(\overline{M},\overline{N})$ is
$R$-flat and furthermore 
for all $X$ in $\mod(R)$ we
have $\Hom_\Dscr(\overline{M},\overline{N}\otimes_R
X)=\Hom_\Dscr(\overline{M},\overline{N})\otimes_R X$.
\end{proposition}
\begin{proof} This is routine. Choose a short exact sequence
\[
0\r Y\r P\r X\r 0
\]
with $P$ a finitely generate projective. We then get an exact sequence
\[
0\r \Hom_\Dscr(\overline{M},\overline{N}\otimes_R Y)\r \Hom_\Dscr(\overline{M},\overline{N}\otimes_R P)\r \Hom_\Dscr(\overline{M},\overline{N}\otimes_R X)\r
\Ext^1_\Dscr(\overline{M},\overline{N}\otimes_R Y)
\]
By the Proposition \ref{ref-5.2-48} and the hypotheses we get $\Ext^1_\Dscr(\overline{M},\overline{N}\otimes_R Y)=0$.
We then obtain a commutative diagram
\[
\xymatrix{
  0\ar[r]& \Hom_\Dscr(\overline{M},\overline{N}\otimes_R Y)\ar[r]& \Hom_\Dscr(\overline{M},\overline{N}\otimes_R P)\ar[r] & \Hom_\Dscr(\overline{M},\overline{N}\otimes_R X)\ar[r]& 0\\
  & \Hom_\Dscr(\overline{M},\overline{N})\otimes_R Y\ar[u]^{\alpha_Y}\ar[r]& \Hom_\Dscr(\overline{M},\overline{N})\otimes_R P\ar@{=}[u]\ar[r]
  & \Hom_\Dscr(\overline{M},\overline{N})\otimes_R X\ar[r]\ar[u]_{\alpha_X}& 0 }
\]
We obtain that $\alpha_X$ is an epimorphism for all $X$. But then $\alpha_Y$ is an epimorphism
from which we then deduce that $\alpha_X$ is an isomorphism. But then $\alpha_Y$
is an isomorphism and hence the lower right exact sequence is in fact exact. Thus
$\Hom_R(\overline{M},\overline{N})$ is $R$-flat, finishing the proof.  
\end{proof}

\begin{thebibliography}{10}

\bibitem{SGA41}
M.~Artin, A.~Grothendieck, and J.~L. Verdier, {\em Th\'eorie des topos et
  cohomologie \'etale des sch\'emas, {SGA4}, {T}ome 2}, Lecture Notes in
  Mathematics, vol. 269, Springer Verlag, 1972.

\bibitem{AZ}
M.~Artin and J.~Zhang, {\em Noncommutative projective schemes}, Adv. in Math.
  {\bf 109} (1994), no.~2, 228--287.

\bibitem{AZ2}
\bysame, {\em {Abstract Hilbert schemes. I.}}, Algebr. Represent. Theory {\bf
  4} (2001), no.~4, 305--394 (English).

\bibitem{Behrend1}
K.~A. Behrend, {\em Derived {$l$}-adic categories for algebraic stacks}, Mem.
  Amer. Math. Soc. {\bf 163} (2003), no.~774, viii+93.

\bibitem{BBD}
A.~Beilinson, J.~Bernstein, and P.~Deligne, {\em Faisceaux pervers},
  Ast{\'e}risque, vol. 100, Soc. Math. France, 1983.

\bibitem{LDeD}
O.~De~Deken and W.~Lowen, {\em Abelian and derived deformations in the presence
  of {Z}-generating geometric helices}, arXiv:1001.4265.

\bibitem{Ek}
T.~Ekedahl, {\em On the adic formalism}, The Grothendieck Festschrift, vol.~2,
  Birkh\"auser, 1990, pp.~197--218.

\bibitem{EGA31}
A.~Grothendieck, {\em \'{E}l\'ements de g\'eom\'etrie alg\'ebrique. {III}.
  \'{E}tude cohomologique des faisceaux coh\'erents. {I}}, Inst. Hautes
  \'Etudes Sci. Publ. Math. (1961), no.~11, 167.

\bibitem{HMS}
D.~Huybrechts, E.~Macri, and P.~Stellari, {\em Formal deformations and their
  categorical general fibre}, arXiv:0809.3201.

\bibitem{Joua}
J.~P. Jouanolou, {\em Syst\`emes projectifs {$J$}-adiques}, Cohomologie
  {$l$}-adique et fonctions {$L$}, {SGA5} (Berlin), Lecture notes in
  mathematics, vol. 589, Springer Verlag, Berlin, 1977.

\bibitem{lowenvdb2}
W.~Lowen and M.~Van~den Bergh, {\em Hochschild cohomology of abelian categories
  and ringed spaces}, Adv. Math. {\bf 198} (2005), no.~1, 172--221.

\bibitem{lowenvdb1}
\bysame, {\em Deformation theory of abelian categories}, Trans. Amer. Math.
  Soc. {\bf 358} (2006), no.~12, 5441--5483.

\bibitem{lowen4}
W.~Lowen, {\em Obstruction theory for objects in abelian and derived
  categories}, Comm. Algebra {\bf 33} (2005), no.~9, 3195--3223.

\bibitem{Polishchuk1}
A.~Polishchuk, {\em Noncommutative proj and coherent algebras}, Math. Res.
  Lett. {\bf 12} (2005), no.~1, 63--74.

\bibitem{VdBSt}
J.~T. Stafford and M.~Van~den Bergh, {\em Noncommutative curves and
  noncommutative surfaces}, Bull. Amer. Math. Soc. (N.S.) {\bf 38} (2001),
  no.~2, 171--216.

\bibitem{VdB28}
M.~Van~den Bergh, {\em Non-commutative {H}irzebruch surfaces},
  arXiv:math/0102005.

\bibitem{VdB26}
\bysame, {\em Non-commutative quadrics}, arXiv:0807.375.

\end{thebibliography}
\def\cprime{$'$} \def\cprime{$'$} \def\cprime{$'$}
\ifx\undefined\bysame
\newcommand{\bysame}{\leavevmode\hbox to3em{\hrulefill}\,}
\fi

\end{document}